\documentclass[11pt]{amsart}
\usepackage{graphicx}
\usepackage{amssymb}

\newtheorem{thm}{Theorem}[section]

\newtheorem{lem}[thm]{Lemma}

\theoremstyle{definition}
\newtheorem{defn}[thm]{Definition}
\theoremstyle{example}

\theoremstyle{remark}

\numberwithin{equation}{section}

\theoremstyle{definition}

\theoremstyle{remark}

\numberwithin{equation}{section}
\newcommand{\h}{\mathcal{H}}

\newcommand{\RR}{\mathbb R}

\newcommand{\SSS}{\mathbb S}
\begin{document}

\title[The Effect of Perturbations of Frames and ...]{The Effect of Perturbations of Frames and Fusion Frames on
Their Redundancies}

% Please mark \corrauth after the name of the corresponding author.
%\title[Multipliers of pg-Bessel sequences ]{Multipliers of pg-Bessel sequences in Banach spaces}
\author[ A. Rahimi, G. Zandi and B. Daraby ]{ A. Rahimi, G. Zandi and B. Daraby  }
\address{ Department of Mathematics, University of Maragheh, P. O. Box 55181-83111, Maragheh, Iran.}

\email{rahimi@maragheh.ac.ir}
\address{ Department of Mathematics, University of Maragheh, P. O. Box 55181-83111, Maragheh, Iran.}
\email{zgolaleh@yahoo.com}
\address{ Department of Mathematics, University of Maragheh, P. O. Box 55181-83111, Maragheh, Iran.}
\email{bdaraby@maragheh.ac.ir}

\begin{abstract}
An interesting question about the perturbed sequences
 is: when do they inherit the properties of the original one?  An elegant relation between  frames (fusion frames) and their perturbations is the relation of their redundancies.
  In this paper, we investigate these relationships. Also, we express the redundancy of frames (fusion frames) in terms of  the  cosine angle between some subspaces.
\end{abstract}

\maketitle

\section{Introduction}
Introduced in 1952 by Duffin and Schaeffer \cite{class}, frames as an extension of the concept of
orthonormal bases allowing for redundancy, while still maintaining  stability properties,
are today an standard notion in mathematics and engineering.  Applications range from more theoretical problems
such as the Kadison- Sineger problem \cite{kadison,casazza} and tensor decomposition \cite{tensor} over questions
inspired by applications such as sparse approximation theory \cite{optimally}, quantum mechanics \cite{tight}, inverse scattering problems \cite{numerical} and wave packet systems on finite cyclic groups, fields and Abelian groups \cite{Gha1,Gha2,Gha3}. Recently, due to both necessities from applications
and theoretical goals, generalizations of this framework have been developed, first fusion frames \cite{6},
then operator-valued frames \cite{operator} and g-frames \cite{g-frames}.
While frames and their extensions are by definition stable in the sense of their analysis operator
being continuous and bounded below, applications require stronger forms of stability. Of particular importance
is robustness with respect to perturbations which is typically regarded as a property of the associated
analysis operators. More precisely, in the frame setting, say one would refer to a frame $\Psi=\{\psi_ i\}_{i\in I}$
as being a $\mu$-perturbation of $\Phi=\{\varphi_i\}_{i\in I},$ if $\|T_\Phi-T_\Psi\|\leq\mu,$ where $T$ is the analysis operator.
In the situation of frames, this is a well-studied subject (see, e.g., \cite{stability, perturbation}).\\
\indent An interesting question about the perturbed sequences
 is: when  do they inherit  the properties of the original one. For instance, it is  known that
 if  $\{\varphi_i\}_{i\in I}$ is a Riesz sequence with lower bound $A,$
 the perturbed sequence $\{\psi_i\}_{i\in I}$ is also a Riesz sequence when $\mu<\sqrt{A},$ \cite{10}.
  An elegant relation between a frame and its perturbations is the relation of their redundancies.
 As we know, redundancy appears as a mathematical concept and as a methodology
for signal processing. Recently, the ability of redundant systems to provide
sparse representations has been extensively exploited \cite{3}. In fact, frame
 theory is entirely based on the notion of redundancy.
 This idea can be interesting to find the relation between
 redundancy of a frames (fusion frames) and the redundancy of its perturbations.
 The notion of the redundancy of a frame were introduced in \cite{2} and
 explained  by authors in \cite{R.Z} for fusion frames. For general frames ( also fusion frames), redundancy should give us information, for instance, about orthogonality and tightness
of the frame, about the maximal number of spanning sets and the minimal number of linearly independent sets our
frame can be divided into, and about robustness with respect to erasures. Ideally, in the case of a unit-norm tight frame,
upper and lower redundancies coincide and equal the customary measure of redundancy $\frac{N}{n} $, where $N$ denotes the number of frame vectors and $n$ the dimension of underlying Hilbert space.
\section{Preliminaries}
In this chapter, we state some of the basic definitions and theorems that are needed in chapter 3.
\subsection{A Review of Frames}
\indent Given a separable Hilbert space $\h$ with norm $\|.\|,$ a sequence
 $\Phi=\{\varphi_i\}_{i\in I},$ where $I$ is a countable index set,  is a frame for $\h$ if there exist
 $0<A\leq B$ such that
 $$A{\|x\|}^{2}\leq\sum_{i\in I}{|\langle x,\varphi_i\rangle |}^{2}\leq B{\|x\|}^{2}$$
 for all $x\in \h.$ The constants $A,B$ are called frame bounds.
 If only the right inequality is satisfied, we say that $\Phi=\{\varphi_i\}_{i\in I}$
 is a \textit{Bessel sequence} with Bessel bound $B.$
 To every Bessel sequence $\Phi=\{\varphi_i\}_{i\in I}$,
 we associate the analysis operator $T_\Phi:\h\longrightarrow {l}^{2}(I)$ defined by
 $T_\Phi x=\{\langle  x, \varphi_i\rangle\}_{i\in I}$
  for $x\in \h,$ and the synthesis operator $T^{*}_{\Phi}:{l}^{2}(I)\longrightarrow \h$
  given by $T^{*}_{\Phi} c=\sum_{i\in I} c_i\varphi_i$ for $c=\{c_i\}_{i\in I}\in {l}^{2}(I).$

  If $\Phi=\{\varphi_i\}_{i\in I}$ is a frame for $\h,$ the frame operator defined
   by
   $$S_{\Phi}:\h\longrightarrow \h, \quad S_{\Phi}x=\sum_{i\in I}
   \langle x,\varphi_i\rangle \varphi_i,\quad x\in\h.$$
   This operator is bounded, positive and invertible.
     A sequence $\{\varphi_i\}_{i\in I}\subseteq \h$ is a
   \textit{Riesz basis} for $\h$ if it is complete in $\h$ and if there exist $0<A\leq B$ such that
   for every finite scalar sequence $c=\{c_i\}_{i\in I}$
   one has
   $$A{\|c\|^{2}}_{l^{2}}\leq{\|\sum_{i\in I}c_i\varphi_i\|}^{2}\leq B {\|c\|^{2}}_{l^{2}},$$
   the constants $A$ and $B$ are called \textit{Riesz bounds}.\\
  \indent In this paper, we will work with perturbed sequences.
    More precisely we will use the following notion of perturbation \cite{p.s}.
\begin{defn}
Let $\Phi=\{\varphi_i\}_{i\in I}$ be a sequence in $\h$
and $\mu>0.$ We say that a sequence $\Psi=\{\psi_i\}_{i\in I}$
in $\h$ is a $\mu$-perurbation of $\Phi=\{\varphi_i\}_{i\in I}$
if  for every finite scalar sequence $c=\{c_i\},$
$$\|\sum c_i(\varphi_i-\psi_i)\|\leq\mu\|c\|_{l^{2}}.$$
\end{defn}
%An interesting question about the perturbed sequences
 %is when it inherits the properties of the original one. For instance, it is  known that
% if  $\{\varphi_i\}_{i\in I}$ is a Riesz sequence with lower bound $A,$
 %the perturbed sequence $\{\psi_i\}_{i\in I}$ is also a Riesz sequence when $\mu<\sqrt{A},$ \cite{10}.
 %An elegant relation between a frame and its perturbations is the relation of their redundancies.
% As we know, redundancy appears as a mathematical concept and as a methodology
%for signal processing. Recently, the ability of redundant systems to provide
%sparse representations has been extensively exploited \cite{3}. In fact, frame
% theory is entirely based on the notion of redundancy.
 The idea of finding the relation between
 redundancies of a frame (fusion frame) and its perturbations should be interesting.
 At the first, we review the definition of the redundancy function for finite frames.

\begin{defn}\cite{2}\label{red.frame}
Let $\Phi=\{\varphi_i\}_{i=1}^{N}$ be a frame for a
finite dimensional Hilbert space ${\h}^{n}$.
For each $x\in \SSS=\{x\in \h: \|x\|=1\}$, the \textit{redundancy function}
$$\mathcal{R}_\Phi:\SSS\to \RR^+$$ is defined by

$$\mathcal{R}_\Phi(x)=\sum_{i=1}^{N}
\|P_{\langle\varphi_i\rangle}(x)\|^2,\quad x\in\SSS$$
where $\langle\varphi_i\rangle$ denotes the span of $\varphi_i \in \h$ and
 $P_{\langle\varphi_i\rangle}$ denotes the orthogonal projection onto $\langle\varphi_i\rangle.$
 The \emph{upper redundancy} of $\Phi$ is defined by $$\mathcal{R}_{\Phi}^+=\sup_{x\in\SSS}
\mathcal{R}_{\Phi}(x), $$ and the \emph{lower redundancy} of
$\Phi$  by $$\mathcal{R}_{\Phi}^-=\inf_{x\in \SSS}
\mathcal{R}_{\Phi}(x). $$ The frame $\Phi$ has \textit{uniform
redundancy} if $ \mathcal{R}_{\Phi}^-=
\mathcal{R}_{\Phi}^+$.
\end{defn}
For more studies and details on the redundancy of frames, we refer to \cite{2}. The following  theorem  states conditions for perturbed
 sequences to be a frame or a Riesz basis \cite{10}.

\begin{thm} \cite{10}\label{1}
Let $\Phi=\{\varphi_i\}_{i=1}^{N}$ be a sequence in $\h$
and assume that  $\Psi=\{\psi_i\}_{i=1}^{N}$ is a $ \mu$-perturbation of
 $\Phi.$ Then the following holds:\\
If $\Phi$ is a frame (Riesz basis) for $\h$ with frame (Riesz) bounds
 $0< A\leq B$ and $ \mu<\sqrt{A},$
then $ \Psi $   is a frame (Riesz basis) for $\h$ with frame (Riesz) bounds
 $ A{(1-\frac{\mu}{\sqrt{A}})^{2}} $
, $ B{(1+\frac{\mu}{\sqrt{B}})^{2}}. $
\end{thm}
\subsection{A Review of Fusion Frames}
In this subsection, we briefly recalling the basic definitions
 and notations of \textit{fusion frames}.
We wish to mention that the definition of a fusion frame and its associated
fusion frame operator already appeared in \cite{6} under the label ``frame of subspaces''.
Besides, we set the definition of the redundancy function for fusion frames
which introduced by authors in \cite{R.Z}.
\begin{defn}
Let $\h$ be a  Hilbert space and $I$ be a (finite or infinite) countable index set. Assume that $\{W_i\}_{i\in
I}$ be a sequence of closed subspaces in $\h$ and $\{v_i\}_{i\in I}$ be a family
of weights, i.e., $v_i>0$ for all $i\in I$. We say that the
family $\mathcal{W}=\{(W_i,v_i)\}_{i\in I}$ is a  $fusion$ $frame$ or a $frame$ $of$ $subspaces$ with
respect to $\{v_i\}_{i\in I}$ for $\h$ if there
exist constants $0<A\leq B<\infty$ such that
$$A\|x\|^2\leq\sum_{i\in I}v_i^2\|P_{W_i}(x)\|^2\leq
B\|x\|^2\quad\forall x\in\h,$$ where $P_{W_i}$ denotes the orthogonal
projection onto $W_i,$ for each $i\in I.$ The fusion frame
$\mathcal{W}=\{(W_i,v_i)\}_{i\in I}$ is called  $tight$  if
$A=B$ and $Parseval$ if $A=B=1$.  If all ${v_i}^{,}s$ take the same value $v$, then $\mathcal{W}$ is
called $v$-$uniform$.
Moreover, $\mathcal{W}$ is called an $orthonormal$ $fusion$
 $basis$ for $\h$ if  $\h= \bigoplus_{i\in I} W_i$.
 If $\mathcal{W}=\{(W_i,v_i)\}_{i\in I}$ possesses an
 upper fusion frame bound but not necessarily a lower
 bound, we call it a $Bessel$ $fusion$ $sequence$ with Bessel fusion bound $B$.
 The normalized version of $\mathcal{W}$ is obtained when we choose $v_i=1$ for all $i\in I.$
 Note that we use this term merely when $\{(W_i,1)\}_{i\in I}$ formes a fusion frame for $\h.$
\end{defn}
Without loss of generality,  we may assume that the family
of weights $\{v_i\}_{i\in I}$ belongs to $\ell^{\infty}_+(I) $.\\
 The analysis, synthesis  and fusion
 frame operator are defined as follows.\\ \par
 \textbf{Notation:} For any family $\{\h_i\}_{i\in I}$ of Hilbert spaces, we use the representation space
 \[(\sum_{i\in I}{\oplus \h_i})_{\ell_2}= \left\{ \{f_{i}\}_{i\in I}: f_i\in\h_i,  \sum_{i\in I}\|f_{i}\|^2<\infty
\right\}\]
with inner product
\[
\langle\{f_i\}_{i\in I},\{g_i\}_{i\in I}\rangle=\sum_{i\in I} \langle f_i, g_i\rangle,\quad \{f_{i}\}_{i\in I}, \{g_{i}\}_{i\in I}\in(\sum_{i\in I}{\oplus \h_i})_{\ell_2}
\]
and
\[
\|\{f_i\}_{i\in I}\|:=\sqrt{\sum_{i\in I} \|f_i\|^2}.
\] It is easy to show that $(\sum_{i\in I}{\oplus \h_i})_{\ell_2}$ is a Hilbert space.

 \begin{defn}
 Let $\mathcal{W}=\{(W_i,v_i)\}_{i\in I}$ be a fusion frame for $\h$.
 The $synthesis$ $operator$ $T_{\mathcal{W}}:(\sum_{i\in I}{\oplus W_i})_{\ell_2} \to \h$ is defined by
  $$  T_{\mathcal{W}}(\{f_i\}_{i\in I})= \sum_{i\in I}{v_if_i}, \quad \{f_i\}_{i\in I}\in (\sum_{i\in I}{\oplus W_i})_{\ell_2}.$$
 In order to map a signal to the representation space, i.e., to Analyze it,
  the $analysis$ $operator$ $T_{\mathcal{W}}^{*}$ is employed which is defined by
  $$T_{\mathcal{W}}^{*}: \h \to (\sum_{i\in I}{\oplus W_i})_{\ell_2} \quad with \quad
  T_{\mathcal{W}}^{*}(f)=\{v_i P_{{W_i}}(f)\}_{i\in I},$$ for any $f\in\h$. The $fusion$ $frame$ $operator$ $S_{\mathcal{W}}$ for $\mathcal{W}$
  is defined by
  $$S_{\mathcal{W}}(f)=T_{\mathcal{W}}T_{\mathcal{W}}^{*}(f)=\sum_{i\in I} v^{2}_i P_{W_i}(f),\quad f\in\h.$$
 \end{defn}
 It follows from \cite{6} that for each fusion frame, the operator $S_{\mathcal{W}}$  is invertible,  positive and $AI\leq S_{\mathcal{W}}\leq BI$. Any $f\in\h$ has the representation $f=\sum_{i\in I} v^{2}_i S_{\mathcal{W}}^{-1}P_{W_i}(f)$.\\
 Now, we set the definition of perturbation of a fusion frame which considered in \cite{Fusion Frames}
 and it was proved that the fusion frames are stable under these perturbations.
 \begin{defn}\cite{Fusion Frames}\label{fusion pert}
Let $\{W_i\}_{i\in I}$ and $\{\tilde {W}_i\}_{i\in I}$
 be families of closed subspaces in ${\h},$
let $\{v_i\}_{i\in I} $ be positive numbers,
 $0\leq {\lambda}_{1},{\lambda}_{2}\lneqq1,$
and $\varepsilon>0.$  If
$$\|(P_{W_{i}}-P_{\tilde {W}_i})f\|\leq \lambda_{1}
\|P_{W_{i}}f\|+ \lambda_{2}\|P_{\tilde{W}_i}f\|+\varepsilon\|f\|\quad f\in \h,$$
then, we say that  $\{(\tilde {W}_i,v_i)\}_{i\in I} $
 is a $({\lambda}_{1},{\lambda}_{2},\varepsilon)$- perturbation
 of $\{(W_i,v_i)\}_{i\in I} .$
\end{defn}
A different notion of perturbation of fusion frames was introduced
in \cite{the effect}, which is equivalent to the above notion of perturbation in finite dimensional case.
\begin{defn}\cite{the effect}\label{fusion pert2}
Let $\mu>0$ and let $\mathcal{W}=\{(W_i,w_i)\}_{i=1}^{N}$ and
 $\mathcal{V}=\{(V_i,v_i)\}_{i=1}^{N}$ be two Bessel fusion sequences
  in $\h.$ We say that $\mathcal{V}$ is a $\mu$-perturbation of $\mathcal{W}$
  (and vice versa) if $(v_iP_{V_i})_{i=1}^{N}$
is a $\mu$-perturbation of $(w_iP_{W_i})_{i=1}^{N}$ or
 $\|T_{\mathcal{W}}-T_{\mathcal{V}}\|\leq\mu$ and so
$$\|w_iP_{W_i}-v_iP_{V_i}\|\leq\mu.$$
\end{defn}
 We state the Proposition 5.2 from \cite{Fusion Frames} for finite dimensional Hilbert space
 ${\h}^{n}$ by assuming that ${\lambda}_{1}={\lambda}_{2}=0$  and
  $\varepsilon=\mu.$
\begin{thm}\cite{Fusion Frames}\label{**}
Let  $\mathcal{W}=\{W_i\}_{i=1}^{N}$ be a fusion frame for
 ${\h}^{n}$ with bounds $A, B$ and $\mu>0$ such that
 $\sqrt{A}-\mu\sqrt{N}>0.$ Further, let $\{(V_i,v_i)\}_{i=1}^{N}$
  be a $\mu$-perturbation of $W$. Then $\{(V_i,v_i)\}_{i=1}^{N}$
   is a fusion frame with fusion frame bounds
 $$(\sqrt{A}-\mu\sqrt{N})^{2}, \quad (\sqrt{B}+\mu\sqrt{N})^{2}.$$
\end{thm}
One of our goals is to investigate the effect of perturbation of fusion frames
on the redundancy.
%As we mentioned, the concept of the redundancy of fusion frames was
%introduced by authors in \cite{R.Z}, which now, we set it.\\
%\indent The redundancy function is defined from the unit sphere $\SSS=\{x\in{\h}^{n}:
%\|x\|=1\}$ to the set of positive real numbers $\RR^+$.
\begin{defn}\cite{R.Z}\label{r.z}
Let $\mathcal{W}= \{(W_i, v_i)\}_{i=1}^{N}$ be a fusion frame for ${\h}^{n}$
with bounds $A$ and $B.$
For each $x\in \SSS$, the $redundancy$ $function$
$\mathcal{R}_{\mathcal{W}}:\SSS\to \RR^+$ is defined by $$
\mathcal{R}_{\mathcal{W}}(x)=\sum_{i=1}^{N}
\|P_{W_i}(x)\|^2,\quad x\in\SSS
.$$

\end{defn}
Notice that this notion is reminiscent of the definition of redundancy function for finite frames, in Definition 1.2, if
$\dim W_i=1$ for $i=1,...,N.$
\begin{defn}
For the fusion frame $\mathcal{W}= \{(W_i, v_i)\}_{i=1}^{N},$ the
\emph{upper redundancy}  is defined by
$$\mathcal{R}_{\mathcal{W}}^+=\sup_{x\in\SSS}
\mathcal{R}_{\mathcal{W}}(x), $$ and the \emph{lower redundancy} of
$\mathcal{W}$  by $$\mathcal{R}_{\mathcal{W}}^-=\inf_{x\in \SSS}
\mathcal{R}_{\mathcal{W}}(x). $$ We say that $\mathcal{W}$ has $uniform$
redundancy if $ \mathcal{R}_{\mathcal{W}}^-=
\mathcal{R}_{\mathcal{W}}^+$.
\end{defn}
%For more details, see \cite{{R.Z}}.
\section{The Effect of Perturbations of Frames on Their Redundancies}
In the definition of the redundancy function for the finite frame
  $\Phi=\{\varphi_i\}_{i=1}^{N}$  the optimal bounds of the
  normalized version of  $\Phi$ considered as the lower and upper redundancies \cite{2}. It is easy to prove the following lemma:
\begin{lem}\label{2}
Let $\Phi=\{\varphi_i\}_{i=1}^{N}$ be a frame for  ${\h}^{n}$
 and   $\Psi=\{\psi_i\}_{i=1}^{N}$ be a
$\mu$-perturbation of $\Phi$ such that  $ \|\varphi_i\|=\|\psi_i\|$
for all $i=1,...,N.$ Then the normalized version of
$\Psi$ is also a $\mu$-perturbation of the normalized version of $\Phi.$
\end{lem}
\begin{proof}
The  sequence $\Psi=\{\psi_i\}_{i=1}^{N}$ is a $\mu$-perturbation of the frame $\Phi,$ so we have
$$\|\sum c_k(\varphi_i-\psi_i)\|\leq \mu {\|c\|}_{l^{2}}$$
for all finite sequence $c=\{c_k\}.$ By assumption, we have $\|\varphi_i\|
=\|\psi_i\|=\alpha_i$ for $i=1,...,N.$ Hence
for all finite scalar sequence  $c=\{c_k\},$
$$\|\sum c_k(\frac{\varphi_i}{\|\varphi_i\|}-\frac{\psi_i}{\|\psi_i\|})\|=
\|\sum c_k(\frac{\varphi_i}{\alpha_i}-\frac{\psi_i}{\alpha_i})\|=
\|\sum c'_k(\varphi_i-\psi_i)\|.$$
By the definition of the perturbation of a frame we have
$$\|\sum c'_k(\varphi_i-\psi_i)\|\leq \mu {\|c'\|}_{l^{2}}.$$
Therefore, the normalized version of  $\Psi$ is a $\mu$-perturbation
of the normalized version of  $\Phi.$
\end{proof}
By using the above lemma, we get the following theorem.
\begin{thm}\label{3}
Let $\Phi=\{\varphi_i\}_{i=1}^{N}$ be a frame for ${\h}^{n}$
with bounds $0<A\leq B$ and $\mu<\sqrt{A}.$
 Let   $\Psi=\{\psi_i\}_{i=1}^{N}$ be a
$\mu$-perturbation of $\Phi$ such that  $ \|\varphi_i\|={\|\psi_i\|}$ for all $i=1,...,N.$
Let $\mathcal{R}_{\Phi}^-$ and  $\mathcal{R}_{\Phi}^+$ denote
the lower and upper redundancies for the frame
$\Phi,$ respectively.
Then $\Psi$ is also a frame for   ${\h}^{n}$and
 the lower and upper redundancies for the frame $\Psi$ are as follows:
$$\mathcal{R}_{\Psi}^-=\mathcal{R}_{\Phi}^-
{(1-\frac{\mu}{\sqrt{\mathcal{R}_{\Phi}^-}})}^{2},
\quad \mathcal{R}_{\Phi}^-=\mathcal{R}_{\Psi}^+
{(1+\frac{\mu}{\sqrt{\mathcal{R}_{\Phi}^+}})}^{2}.$$
Moreover, if $\Phi=\{\varphi_i\}_{i=1}^{N}$ is a Riesz basis, then
$$\mathcal{R}_{\Phi}^-=\mathcal{R}_{\Psi}^-=\mathcal{R}_{\Psi}^+=\mathcal{R}_{\Phi}^+=1.$$
\end{thm}
\begin{proof}
Let $\Phi=\{\varphi_i\}_{i=1}^{N}$ be a frame for ${\h}^{n}$
with bounds $0<A\leq B$ and $\mu<\sqrt{A}.$
 Let   $\Psi=\{\psi_i\}_{i=1}^{N}$ be a
$\mu$-perturbation of $\Phi.$
Then by Theorem \ref{1},  $\Psi$ is also a frame for   ${\h}^{n}$ with bounds
$$ A{(1-\frac{\mu}{\sqrt{A}})^{2}} \quad
,  B{(1+\frac{\mu}{\sqrt{B}})^{2}}.$$
 By assumption, we have $\|\varphi_i\|
=\|\psi_i\|=\alpha_i$ for $i=1,...,N.$ Then by Lemma \ref{2},
the normalized version of
$\Psi$ is also a $\mu$-perturbation of the normalized version of $\Phi.$
We know that the redundancy function \cite{2} for a frame
 consider the normalized version of the frame and lower and upper redundancies coincide
 the optimal lower and upper frame bounds of the normalized version of the original frame. So
 the frame bounds of the frame $\{\frac{\varphi_i}{\|\varphi_i\|}\}_{i=1}^{N}$
 equal to $\mathcal{R}_{\Phi}^-$ and $\mathcal{R}_{\Phi}^+.$
 Similar statement holds for redundancies of the frame  $\Psi.$
 Hence , we conclude
 $$\mathcal{R}_{\Psi}^-=\mathcal{R}_{\Phi}^-
{(1-\frac{\mu}{\sqrt{\mathcal{R}_{\Phi}^-}})}^{2},
\quad \mathcal{R}_{\Psi}^-=\mathcal{R}_{\Phi}^+
{(1+\frac{\mu}{\sqrt{\mathcal{R}_{\Phi}^+}})}^{2}.$$
 For the
  moreover part, we know \cite{2} that the redundancy of each Riesz basis equals to one. Hence by Theorem \ref{1} the sequence
     $\Psi=\{\psi_i\}_{i=1}^{N}$ is a Riesz basis, so
$$\mathcal{R}_{\Phi}^-=\mathcal{R}_{\Phi}^+=
\mathcal{R}_{\Psi}^-=\mathcal{R}_{\Psi}^+=1.$$
\end{proof}
By extra hypothesists $\mathcal{R}_{\Phi}^-,\mathcal{R}_{\Phi}^+<\infty$ on the lower and upper
 redundancies of
the frame $\Phi,$ above theorem extended to the infinite case.
We can express the lower and upper redundancies of a
frame in terms of cosine angle between some subspaces.
At the first we review the definition of sine and cosine between subspaces \cite{p.s}.
\begin{defn}\cite{p.s}\label{sine}
Assuming that $V$ and $W$ are subspaces of $\h$ and
 $V\neq 0.$ The angle from $V$ to $W$
is defined as the unique number $\theta(V,W)\in [0,\frac{\pi}{2}]$ for which
\begin{equation}\label{e1}
\operatorname{cos}\theta(V,W)=\inf_{{f\in V},{ \|f\|=1}}\|P_W(f)\|.
\end{equation}

We will frequently use the notion
$$R(V,W)=\operatorname{cos}\theta(V,W),$$which is
 called the infimum cosine angle.
If we take the supremum  instead of the infimum of the right hand side of (\ref{e1}),
we obtain the supremum cosine angle $S(V,W)$ of $V$ and $W,$ that is
$$ S(V,W)=\sup_{{f\in V},{ \|f\|=1}}\|P_W(f)\|. $$
$R$ and $S $ are related by
$$R(V,W)=(1-{S}^{2}(V,{W}^{\perp}))^{\frac{1}{2}}.$$
The \textit{gap} of subspace $V$ to $W$
 is defined by
 $$\delta(V,W)=\sup_{{x\in V},{ \|x\|=1}}dist (x,W)=
 \sup_{{f\in V},{ \|f\|=1}}\inf_{y\in W}\|x-y\|.$$
 An elementary calculation shows that the gap of $V$ to $W$ are related via
 $$\delta(V,W)=\operatorname{sin}\theta(V,W).$$
\end{defn}

By the properties of  \textit{inf} and  \textit{sup} and applying
the definitions of infimum cosine and supremum cosine  in the
 Definition \ref{sine}, we  have this elegant relations:\\
$$\mathcal{R}_{\Phi}^-=\sum_{i=1}^{N
}R^{2}(W',\langle\varphi_i\rangle),\quad
\mathcal{R}_{\Phi}^+=\sum_{i=1}^{N}
S^{2}(W',\langle\varphi_i\rangle),$$
\\where $W'$ is a subspace of $\h$ which contains the unit sphere $\SSS$, i.e., $\SSS\subseteq W'.$
Therefore in the Theorem \ref{3}, we have
$$\mathcal{R}_{\Psi}^-=\sum_{i=1}^{N
}R^{2}(W',\langle\varphi_i\rangle) (1-\frac{\mu}
{\sqrt{\sum_{i=1}^{N}R^{2}(W',\langle\varphi_i\rangle)}})^{2},$$
and for upper redundancy,
$$\mathcal{R}_{\Psi}^+=\sum_{i=1}^{N}
{S}^{2}(W',\langle\varphi_i\rangle){(1+\frac{\mu}{\sqrt{\sum_{i=1}^{N}S^{2}(W',\langle\varphi_i\rangle)}})}^{2}.$$
The above relations can be expressed in terms of the gap between
 the subspaces $W'$ and
$\langle\varphi_i\rangle,$
because we have
$$\delta (W',\langle \varphi_i \rangle)=\sqrt{1-{R}^{2}
(W',\langle \varphi_i\rangle)}.$$
\section{The Effect of Perturbations of Fusion Frames on Their Redundancies}

In the definition of the redundancy function for fusion frames \cite{R.Z},
 we consider lower and upper bounds of the normalized version of a fusion frame
 as lower and upper redundancies.

In the following theorem we investigate the relationship between
the redundancy of a fusion frame and the redundancy of its perturbations.
 We assume that each subspace has a weight equal 1.
\begin{thm}
Let  $\mathcal{W}=\{W_i\}_{i=1}^{N}$ be a fusion frame for
 ${\h}^{n}$ with bounds $A, B$ and $\mu>0$ such that
 $\sqrt{A}-\mu\sqrt{N}>0.$ Let $\mathcal{V}=\{V_i,\}_{i=1}^{N}$
  be a $\mu$-perturbation of $\mathcal{W}.$ Then $\mathcal{V}$
   is a fusion frame for  ${\h}^{n}$ and the following relations
 between lower and upper redundancies of $\mathcal{W}$ and
$\mathcal{V}$
hold:
$$\mathcal{R}_{\mathcal{V}}^-=(\sqrt{\mathcal{R}_{\mathcal{W}}^-}-\mu\sqrt{N})^{2},\quad
\mathcal{R}_{\mathcal{V}}^+=(\sqrt{\mathcal{R}_{\mathcal{W}}^+}+\mu\sqrt{N})^{2}. $$
\end{thm}
\begin{proof}
 The family $\mathcal{V}$ is a  $\mu$-perturbation of the fusion frame $\mathcal{W}.$ Hence by
Theorem \ref{**}, $\mathcal{V}$ is also a fusion frame for  ${\h}^{n}$.
The subspaces in fusion frames  $\mathcal{W}$ and $\mathcal{V}$ have weights equal to 1, so they
are normalized fusion frames and therefore the optimal bounds are equal to lower and upper
redundancies. Theorem \ref{**} implies:
 $$\mathcal{R}_{\mathcal{V}}^-=(\sqrt{\mathcal{R}_{\mathcal{W}}^-}-\mu\sqrt{N})^{2},\quad
\mathcal{R}_{\mathcal{V}}^+=(\sqrt{\mathcal{R}_{\mathcal{W}}^+}+\mu\sqrt{N})^{2}. $$
\end{proof}
Let  $\mathcal{W}=\{W_i\}_{i=1}^{N}$ be a fusion frame for
${\h}^{n}.$ The lower and upper redundancies of $\mathcal{W}$ can be express as follows:\\
Let $W'$ be a subspace of ${\h}^{n}$ such that involves
the unit sphere $\SSS,$ i.e., $\SSS\subseteq W'.$
Then
$$\mathcal{R}_{\mathcal{W}}^-=\sum_{i=1}^{N}R^{2}(W', W_i),
 \quad \mathcal{R}_{\mathcal{W}}^+=\sum_{i=1}^{N}S^{2}(W', W_i) .$$
\\
\textbf{Acknowledgments:} The authors would like to thank referee(s) for valuable comments and suggestions.

\end{document}